\documentclass[11pt,twoside]{amsart}
\usepackage{amssymb,amsmath,amsthm}
\usepackage[mathscr]{euscript}
\usepackage{xcolor}             

\newcommand{\floor}[1]{\lfloor{#1}\rfloor}

\newcommand{\N}{{\mathbb N}}
\newcommand{\C}{{\mathbb C}}

\newcommand{\Z}{{\mathbb Z}}

\newcommand{\vep}{\varepsilon}
\newcommand{\mV}{\mathcal{V}}

\newcommand{\mfteib}{N^{\vartriangleleft}\left(F,\varepsilon\right)\cap I}
\newcommand{\mfteiba}{N^{\vartriangleleft}\left(F,\varepsilon,a\right)}

\newcommand{\Q}{{\mathbb Q}}
\newcommand{\R}{{\mathbb R}}
\newcommand{\T}{{\mathbb T\,}}

\newcommand{\J}{\mathfrak{J}}
\font\seis=cmr6

\def\wap{{\seis{\mathscr{WAP}}}}

\def\ap{{\seis{\mathscr{AP}}}}

\newcommand{\g}{{{G}}}

\newcommand{\stl}{\mathrm{StL}}

\newtheorem{definition}{Definition}[section]

\newtheorem{theorem}[definition]{Theorem}

\newtheorem{lemma}[definition]{Lemma}
\newtheorem{corollary}[definition]{Corollary}
\newtheorem{proposition}[definition]{Proposition}
\theoremstyle{definition}

\def\Z{{\mathbb Z}}
\def\R{{\mathbb R}}
\def\C{{\mathbb C}}
\def\N{{\mathbb N}}
\def\P{{\mathbb P}}
\def\T{{\mathbb T}}
\def\emptyset{\varnothing}

\def\mapsto{\mapstochar\longrightarrow}

\newcommand\restr[2]{{
  \left.\kern-\nulldelimiterspace 
  #1 
  \vphantom{\big|} 
  \right|_{#2} 
  }}

\title{\bf Sampling Almost Periodic and related  Functions}
\author{Stefano Ferri \and Jorge Galindo \and Camilo G\'omez}
\thanks{The support of  research project P1$\cdot$1B2014-35 (Universitat
Jaume I) is gratefully acknowledged. \\
Research of  the first  named  author  supported by
the Faculty of Sciences of Universidad de los Andes via the
\emph{Proyecto Semilla\/}: ``Representabilidad de grupos topol\'ogicos y de
\'Algebras de Banach y aplicaciones''.\\
Research of  the second  named  author supported by Ministerio
de Economía y Competitividad (Spain) through project MTM2016-77143-P (AEI/FEDER, UE).\\
The support is gratefully acknowl\-edged.
}

\address{\noindent Stefano Ferri,
Departamento de Matem\'aticas,
Universidad de los Andes, Cra 1 18A--10, Bogot\'a D.C.,
Colombia. Apartado A\'ereo 4976.
\hfill\break \noindent E-mail: {\tt stferri@uniandes.edu.co}}
\address{\noindent Jorge Galindo, Instituto Universitario de Matem\'aticas y
Aplicaciones (IMAC)\\ Universidad Jaume I, E--12071, Cas\-tell\'on,
Spain. \hfill\break \noindent E-mail: {\tt jgalindo@uji.es}}
\address{\noindent Camilo G\'omez, Departamento de Matem\'aticas,
Universidad de los Andes, Cra 1 18A--10, Bogot\'a D.C.,
Colombia.  Apartado A\'ereo 4976. Facultad de Ingenier\'ia,
Universidad de La Sabana, Campus Universitario Puente del Com\'un, Ch\'ia, Colombia. Departamento de Matem\'aticas, Escuela Colombiana de Ingenier\'ia, AK 45 205--59, Bogot\'a D.C.,
Colombia.
\hfill\break \noindent E-mail: {\tt alfon-go@uniandes.edu.co}}

\date{Version of the 12th of December 2018}
\begin{document}
\pagestyle{myheadings}

\begin{abstract}
\par
We consider certain finite sets of circle-valued functions defined on
intervals of  real numbers and  estimate how large the  intervals must
be  for   the values of these functions  to be uniformly distributed
in an approximate way.
 This is used to establish
some general conditions under which  a random construction introduced by
Katznelson for the integers yields sets that are dense in the Bohr group.
We obtain in this way very sparse sets of real numbers (and of integers) on which
two different  almost periodic  functions cannot agree, what makes them
amenable to be used in sampling theorems for these functions. These sets
can be made as sparse as to have zero asymptotic density or  as  to be t-sets,
i.e., to be sets that intersect any of their translates in a bounded
set. Many of these results  are proved not only for almost periodic
functions but also for classes of functions generated by more general
complex exponential functions, including chirps or polynomial phase functions.
\par
\noindent{\it File\/}: {\tt FerriGalindoGomezAPSamplingSecondSubmit.tex}

\end{abstract}
\maketitle
\section{Introduction}
Many sampling processes depend  on choosing a sampling set where
the functions to be sampled are uniquely determined. In the case of
almost periodic functions on $\R$, sets with that property admit  a
neat topological description: they are precisely those subsets of $\R$
that are  dense in the Bohr topology.
A clear separation between consecutive samples is another natural
prerequisite and, for this reason, sampling sets  are  usually required to
be  uniformly discrete. Although Bohr-density and  uniform discreteness
might strike as  conflicting requirements, they  are not
totally incompatible: Bohr-dense sets can be not only uniformly discrete
but even quite sparse, as we next describe.

Collet \cite{coll96} (see also Carlen and Mendes \cite{carlemend09}
for a similar approach to functions that can be approximated by
polynomial phase functions) proved that a random
selection of  points  in  regularly spaced time-windows almost surely
produces a set that is uniformly distributed in the Bohr compactification,
hence dense in the Bohr topology (see the section  below for unexplained
terms), and whose asymptotic density is as small  as desired.

Motivated by different sort of problems, Katznelson \cite{katz73}
devised a method that almost surely produces subsets of the group of integers $\Z$
with asymptotic null density which are dense in the Bohr topology. These
constructions were further  developed  in \cite{katzkaha08,liqueffrodr02}.

In this paper we lay out general conditions for Katznelson's method
to work both in $\R$ and $\Z$ for almost periodic functions and, more
generally, for function spaces generated by complex exponentials
$e^{2\pi i p(t)}$ with $p(t)$ running over sets of polynomials  of
bounded degree with coefficients
in $\R$ (usually known as \emph{chirps} or \emph{polynomial phase functions}).
Namely, we prove that partitioning $\R$ into intervals of increasing length
and then choosing $\ell_k$ random points in each partition,
with $\ell_k$ larger than the order of the logarithm of the size of the corresponding
cell of the partition, we  almost surely obtain a set where these
functions are uniquely determined, a set of uniqueness for them, see
Definition \ref{def:uniq}.

The dense sets we obtain, as those in \cite{carlemend09,coll96,katz73},
are obtained through the Borel--Cantelli Lemma and might require extremely large
intervals to be reliable. Our approach gives hints on the minimum
size a sampling interval must be in order to use our methods for
approximate sampling.

Once our  general construction
is laid,  dense subsets  of the Bohr group with specific properties
are easy to obtain.
In order to illustrate this, we show the existence of
t--sets that are dense in the Bohr compactification.  By a t--set
we refer here to  a special sort of \emph{thin} sets, introduced by
Rudin \cite{rudin59}, which are sets of
interpolation for the weakly almost periodic functions, see Section
\ref{t-set} for more on this.

\subsection{Notation and terminology}\label{notationandterminology}

Even if most of our results are proved  for $G=\R$ or $G=\Z$, we find
more convenient to state the  results of Section \ref{sec:appdense}
for general topological Abelian groups or  for locally compact Abelian (LCA)
groups, depending on whether they  involve probabilities  or not. In
this latter case  probabilities will be computed using the (essentially
unique) invariant measure $\lambda_G$ that every locally compact group
$G$ carries.
When it comes to $G=\R$ and $G=\Z$ we assume that Haar measures are
normalized so that $\lambda_G$ is the Lebesgue measure.
If  $I\subseteq \R$  is an interval, $\lambda_\R(I)$ therefore
corresponds to the length of $I$.

In many of our proofs, we consider a random choice  $\Lambda$ of,
say, $\ell$--many elements of a subset $I$ of a given locally
compact group $G$ and we estimate the probability that
$\Lambda$  belongs to a measurable subset $\mathcal{A}$ of $I^\ell$
that is invariant under permuting coordinates. This random choice,
and its corresponding probability, is always to be understood in the
probability space  induced by $\lambda_{\g^\ell}$ on $I^\ell$. Hence
for a given $\mathcal{A}\subseteq I^\ell$ we have that
$\P\left(\left\{\Lambda\colon \Lambda\in \mathcal{A}\right\}\right)=
\lambda_{\g^\ell}( \mathcal{A})/\lambda_\g(I)^\ell$.

When it comes to duality, the circle group $\T$ has a central r\^ole. When
a specific distance on $\T$ is required,  our choice is   the \emph{angular
distance} defined for every $t,\,s\in [0,1)$ by
\[
d_{a}\left(e^{2 \pi i t},e^{2\pi is}\right)=\min\{|t-s|,1-|t-s|\}.
\]
The open ball of radius $\varepsilon$ and center 1 will be denoted as
$\mV_{\vep}$, thus:
\[ \mV_{\vep}=\left\{e^{2\pi i t}\in \T \colon |t|<\vep\right\}\subseteq  \T.\]

If $G$ is a topological group, $C(G,\T)$ denotes the multiplicative
group of continuous $\T$--valued functions, and $\widehat{G}$
its subgroup of continuous homomorphisms. We refer to
$\widehat{G}$ as the \emph{character group} of $G$.
Characters of $\R$ are denoted by $\chi_\tau$, $\tau \in \R$
where $\chi_\tau(s)=e^{2\pi i \tau s}$, for every $s\in \R$. It is
a standard fact that the map $\tau\mapsto \chi_{\tau}$ establishes
an isomorphism which maps $\R$ onto $\widehat{\R}$. The same mapping with
$\tau\in \Z$ establishes an  isomorphism which maps $\Z$ onto $\widehat{\T}$.

If $p(x)$ is a polynomial with real coefficients we denote by $\psi_p$
the function $\psi_p(t)=e^{2 \pi i p(t)}$, so that
$\psi_p=\chi_\tau$, when $p(t)=\tau t$.
The symbol $\mathfrak{C}_n$ stands for the set
$\big\{e^{2\pi i p(t)}\colon p\in \R_n[x]\big\}$, where $\R_n[x]$ denotes
the set of polynomials in $\R[x]$ of degree at most $n$.
Functions in $\mathfrak{C}_n$ are known as \emph{chirps} or \emph{polynomial phase functions
}.

We regard almost periodic functions as functions that can be
approximated by linear combinations of characters (known as trigonometric
polynomials). In the next definition we extend this approach to
spaces generated by other $\T$--valued functions, see \cite{bohr47} or
\cite{katz}, for other definitions, including H. Bohr's original one in
terms of translation numbers (called \emph{almost periods} in recent references).

\begin{definition}\label{def:a(g)}
Let $G$ be a topological group.
 For $\mathfrak{J}\subseteq C(G,\T)$, let $\mathrm{span}(\mathfrak{J})$ denote the linear span of $\mathfrak{J}$ in the vector space $C(G,\C)$. We define
 \begin{align*}
\ap_{\mathfrak{J}}(G)&=\overline{\mathrm{span}(\mathfrak{J})}^{\|\cdot\|_\infty}.\end{align*}
When $\mathfrak{J}=\widehat{G}$ we obtain the space of  \emph{almost
periodic} functions on $G$  and we simply denote it as $\ap(G)$.
\end{definition}
The topology that almost periodic functions induce on a group is known
as the \emph{Bohr topology}. This is the topology the group inherites
from its embedding in $\prod_{\chi \in \widehat{G}}\T_\chi$, with
$\T_\chi=\T$ for every $\chi$,  given by
$g\mapsto (\chi(g))_{\chi}$
for every $g\in G$.
The closure $G^{\ap}$
of $G$ in $\prod_{\chi \in \widehat{G}}\T_\chi$ is known as the Bohr
compactification of $G$ and can also be identified with the spectrum of
$\ap(G)$ when $\ap(G)$ is viewed as a Banach algebra.

A bounded and continuous function $f$ on $G$ is almost periodic precisely when
it can be continuously extended to a function $f^b\in
C(G^{\ap})$. Therefore if $D$ is a discrete subset of $G$ which is  dense
in $G^{\ap}$, then there is at most one almost periodic function $f$
on $G$ that fits any  values that were preassigned on $D$.

We next define the sort of almost periodic  functions that are most
suitable for our sampling methods, the functions with  summable
Bohr--Fourier series.  \begin{definition}\label{ap1(g)} Let $G$ be a
topological group and  $\mathfrak{J}\subseteq C(G,\T)$. We define
\begin{align*}
A_{\mathfrak{J}}(G)&=\biggl\{\sum_{\phi \in \mathfrak{J}}
\alpha_{_{\phi}}\phi \colon  \sum_{\phi \in \mathfrak{J}}\left|
\alpha_{_{\phi}}\right|<\infty\biggr\}.
\end{align*}
The natural isomorphism between the Banach algebra $\ell_1(\mathfrak{J})$
and  $A_{\mathfrak{J}}(G)$ defines a norm on this latter space. We denote
this norm by $\|\cdot\|_{A_{\mathfrak{J}}}$.
\end{definition}
When $\mathfrak{J}=\widehat{G}$, $A_{\mathfrak{J}}(G)$ can  be identified
with the  Fourier algebra  $A(G^\ap)$ (as defined, for instance, in
\cite[Section 1.2.3]{rudin90})  on  the Bohr compactificaction $G^\ap$
of $G$.  Note  that  $A_{\mathfrak{J}}(G)$ may be  strictly  contained
in $\ap_{\mathfrak{J}}(G)$, see \cite[Theorem 4.6.8]{rudin90}, the result
is originally due to Segal \cite{segal50}.

If $X$ is a set, a subset $A\subseteq X$ is an \emph{$n$--subset} if its
cardinality is $n\in \N$. A subset $D$ of a metric space $(X,\rho)$
is \emph{$\varepsilon$-dense} if for every $x\in X$ there exists $d_x\in
D$ such that $\rho(x,d_x)<\varepsilon$.

\section{Approximately Bohr--dense subsets}\label{sec:appdense}

We aim to construct sets where particular families of functions are uniquely determined.
\begin{definition}\label{def:uniq}
Let $\Lambda$ be a subset of a topological group $G$  and let  $\mathcal{A}$ be a
vector subspace of $C(G,\C)$. We say that $\Lambda$ is a \emph{set of
uniqueness} for $\mathcal{A}$ if  whenever  $f\in \mathcal{A}$ satisfies
$\restr{f}{\Lambda}=0$, then $f=0$.
\end{definition}
We first consider sets of approximate uniqueness.
\begin{definition}\label{def:ss}
Let $G$ be a topological group. Fix $I\subseteq G$ and
$\mathfrak{J}\subseteq C(G,\T)$. The subset $\Lambda\subseteq I$ is an
\emph{$(\J,I,\vep)$--sampling set} if
for every  $f\in A_{\mathfrak{J}}(G)$,
\[ \left\|\restr{f}{I}\right\|_{\infty}\leq \vep\cdot \|f\|_{A_{\mathfrak{J}}}+
\left\|\restr{f}{\Lambda}\right\|_\infty.\]
\end{definition}
The following sets contain enough elements to approximate the values of a
given family $F$ of functions in $C(G,\T)$.
\begin{definition}\label{def:matc}
Let $G$ be a topological group. Fix $F\subseteq C(G,\T)$, $I\subseteq G$
and $\vep>0$. The subset  $\Lambda\subseteq G$ is
an \emph{$(F,I,\vep)$--matching set} if for every
$a\in I$ there exists $x_a\in \Lambda$ such that
$\phi(x_a)\in \phi(a) \cdot\mathcal{V}_{\vep}$ for every $\phi\in F$.
\end{definition}
Notation below provides a compact expression of the property defining matching sets.
\begin{definition}\label{polars}
For $I\subseteq G$, $a\in G$, $\Delta\subseteq C(G,\T)$ and $\varepsilon >0$, we define
\begin{enumerate}
\item ${\displaystyle N^{\vartriangleright}(I,\varepsilon)=\left\{\phi\in
C(G,\T)\colon \phi(I)\subseteq \mathcal{V}_\varepsilon\right\}}$, and
\item ${\displaystyle N^{\vartriangleleft}(\Delta,\varepsilon,a)=\left\{t\in
G \colon \phi(t)\in \phi(a)\cdot\mathcal{V}_\varepsilon
\mbox{ for all } \phi\in \Delta\right\}}$. For $a=0$ we write
$N^{\vartriangleleft}(\Delta,\varepsilon)$.
\end{enumerate}
\end{definition}
Observe that $N^{\vartriangleleft}(\Delta,\varepsilon,a)=
a+N^{\vartriangleleft}(\Delta,\varepsilon)$ in case $\Delta$ consists of homomorphisms.
In these terms, the subset $\Lambda\subseteq G$ is an
\emph{$(F,I,\vep)$-matching set} if $\Lambda \cap \mfteiba \neq \emptyset$
for every $a\in I$.
We now see that matching sets, even when slightly thickened,  are sampling sets.
\begin{proposition}\label{mat->samp}
Let $G$ be a topological  group, $\Lambda\subseteq I\subseteq G$
and $\mathfrak{J}\subseteq C(G,\T)$. Suppose that there exists $F\subseteq C(G,\T) $ that
satisfies
\begin{enumerate} \item $\Lambda$ is an
$(F,I,\vep)$--matching set, and \item $\mathfrak{J}\subseteq F\cdot
N^{\vartriangleright}\left(I,\varepsilon\right)$.
\end{enumerate}
Then $\Lambda$ is a $(\J,I,(6\pi+2)\vep)$--sampling set.
\end{proposition}
\begin{proof}
For $f\in A_{\mathfrak{J}}(G)$, let $(\phi_j)_{j=1}^n\subseteq \mathfrak{J}$
and $(\alpha_j)_{j=1}^n\subseteq\C$ be such that
$\big\|f-\sum_{j=1}^n \alpha_j\phi_j\big\|_\infty<\varepsilon\cdot
\left\|f\right\|_{A_{\mathfrak{J}}}$. Let $z\in
I$ be fixed. Given that $\mathfrak{J}\subseteq F\cdot
N^{\vartriangleright}\left(I,\varepsilon\right)$, we can choose $\psi_j\in
F$ and $\kappa_j\in N^{\vartriangleright}\left(I,\varepsilon\right)$
with $\phi_j=\psi_j\cdot \kappa_j$. Since $\kappa_j \in
N^{\vartriangleright}\left(I,\varepsilon\right)$ implies
$|\kappa_j(z)-1|<2 \pi \varepsilon$, we obtain
\begin{equation*}
\Bigg|\sum_{j=1}^n \alpha_j\phi_j(z)-\sum_{j=1}^n
\alpha_j\psi_j(z)\Bigg|\leq \sum_{j=1}^n |\alpha_j|\cdot
\left|\psi_j(z)\right|\cdot \left|\kappa_j(z)-1\right|<2\pi
\varepsilon\cdot \|f\|_{A_{\mathfrak{J}}}.  \end{equation*} On the
other hand, since $\Lambda$ is an $(F,I,\varepsilon)$--matching set,
there exists $x_z\in \Lambda$ such that $\psi(x_z)\in \psi(z)\cdot
\mathcal{V}_\varepsilon$ for every $\psi\in F$. Hence
\begin{equation*}
\Bigg|\sum_{j=1}^n \alpha_j\psi_j(z)-\sum_{j=1}^n
\alpha_j\psi_j(x_z)\Bigg|\leq \sum_{j=1}^n
|\alpha_j|\cdot\left|\psi_j(z)-\psi_j(x_z)\right|\leq 2\pi
\varepsilon\cdot \|f\|_{A_{\mathfrak{J}}}.
\end{equation*}
Therefore, from
\begin{align*}
\left|f(z)\right|&\leq \Bigg|f(z)-\sum_{j=1}^n
\alpha_j\phi_j(z)\Bigg|+\Bigg|\sum_{j=1}^n \alpha_j\phi_j(z)-\sum_{j=1}^n
\alpha_j\psi_j(z)\Bigg|\\
&+\Bigg|\sum_{j=1}^n \alpha_j\psi_j(z)-\sum_{j=1}^n
\alpha_j\psi_j(x_z)\Bigg|+\Bigg|\sum_{j=1}^n
\alpha_j\psi_j(x_z)-\sum_{j=1}^n \alpha_j\phi_j(x_z)\Bigg|\\
&+\Bigg|\sum_{j=1}^n \alpha_j\phi_j(x_z)-f(x_z)\Bigg|+\left|f(x_z)\right|,
\end{align*}
we obtain
\begin{align*}
|f(z)|&<\varepsilon\cdot \|f\|_{A_{\mathfrak J}}+2\pi\varepsilon\cdot
\|f\|_{A_{\mathfrak J}}
+2\pi\varepsilon\cdot \|f\|_{A_{\mathfrak J}}+2\pi\varepsilon\cdot \|f\|_{A_{\mathfrak J}}
+\varepsilon\cdot \|f\|_{A_{\mathfrak J}}+\left\|\restr{f}{\Lambda}\right\|_\infty\\
&=(6\pi +2)\varepsilon\cdot \|f\|_{A_{\mathfrak
J}}+\left\|\restr{f}{\Lambda}\right\|_\infty.
\end{align*}
\end{proof}
As it may be expected,  Proposition \ref{mat->samp} yields density
results when  applied to all finite subsets of a concrete subspace of
$C(G,\T)$.
\begin{corollary}\label{mat->dense}
Let $G$ be a
topological group, $\mathfrak{J}\subseteq C(G,\T)$  and  $\Lambda\subseteq
G$. Suppose that $\Lambda$ is an $(F,G,\varepsilon)$--matching set for
every finite subset $F\subseteq \J$   and every $\varepsilon>0$. Then $\Lambda$ is a  set of uniqueness for $\ap_{\mathfrak{J}}(G)$.
\end{corollary}
\begin{proof}
Let $f\in \ap_{\mathfrak J}(G)$ with $\restr{f}{\Lambda}=0$ and let
$0<\varepsilon<1$ be fixed. Find  a $\mathfrak{J}$--trigonometric
polynomial $P_f=\sum_{j=1}^n \alpha_j \psi_{j}$, with
$(\alpha_j)_{j=1}^n\subseteq \C$ and $(\psi_j)_{j=1}^n\subseteq
\mathfrak{J}$, such that $\left\|f-P_f\right\|_\infty<\varepsilon$.
This implies that $\left\|\restr{P_f}{\Lambda}\right\|_{\infty}\leq \varepsilon$. With $
\tilde{\varepsilon}=\varepsilon/\sum_{j=1}^{n}|\alpha_j|$ and
$F=\{\psi_1,\dots,\psi_n\}$ we obtain from Proposition \ref{mat->samp}
that $\Lambda$ is an $(F,G,(6\pi+2)\tilde{\varepsilon})$--sampling
set. Given that $P_f\in A_{F}(G)$, we have as a consequence
that
\[
\|P_f\|_\infty\leq  (6\pi+2)\tilde{\varepsilon}\cdot
{\sum_{j=1}^{n}|\alpha_j|}+\varepsilon\leq (6\pi +3) \varepsilon.
\]
Since $\|f-P_f\|_\infty\leq \varepsilon$, we deduce that
$\|f\|_\infty<(6\pi +4)\varepsilon$. We conclude that $f=0$, for
$\varepsilon$ was arbitrary.
\end{proof}
In the case of almost periodic
functions, every continuous function on $G^\ap$  coincides with an almost
periodic function on $G$, we obtain therefore that $\Lambda$ is dense in
$G^\ap$.
\begin{corollary}\label{cor:mat->dense} Let $G$ be an Abelian
topological group and  $\Lambda\subseteq G$. Suppose that $\Lambda$ is
an $(F,G,\varepsilon)$-matching set for every finite subset $F\subseteq
\widehat{G}$   and every $\varepsilon>0$. Then  $\Lambda$ is a set of
uniqueness for $ \ap(G)$. In particular $\Lambda$ is dense in $G^\ap$.
\end{corollary}
\begin{proof}
The only difference with Corollary \ref{mat->dense} resides in the
density statement. The Gelfand Representation identifies $\ap(G)$
with $C(G^\ap,\C)$. If $\Lambda$ is not dense in $G^\ap$, by Urysohn's
Lemma there would be a nonconstant  function $f\colon G^\ap \to \C$
that vanishes on $\Lambda$ which is impossible since $f$ is determined
by its values on $\Lambda$.
\end{proof}

To be able to cover simultaneously the cases of $\R$ and $\Z$
we state our next results in the context of locally compact groups with
Haar measure.

Our next objective is to estimate the probability of selecting a set that is  $(F,\{a
 \},\varepsilon)$-matching set for a fixed $a\in G$.

We start by introducing some notation.
\begin{definition}\label{ui} For $I\subseteq G$, $n\in \N$, $a\in G$ and $\varepsilon >0$, we define
\begin{align*}
\mathfrak{P}_{I,a,n,\varepsilon}&=\left\{F\subseteq
C(G,\T)\colon |F|=n\mbox{ and } \lambda_\g\left(
N^{\vartriangleleft}(F,\varepsilon,a)\cap I\right) \geq \varepsilon^n
\lambda_\g(I)\right\}, \text{\ and}\\
\mathfrak{P}_{I,n,\varepsilon}&=\left\{F\subseteq
C(G,\T)\colon |F|=n\mbox{ and } \lambda_\g\left(
N^{\vartriangleleft}(F,\varepsilon,a)\cap I\right)
\geq \varepsilon^n \lambda_\g(I) \mbox{ for all } a\in
G\right\}.\end{align*} \end{definition}
The sets $\mathfrak{P}_{I,a,n,\varepsilon}$ are formed by
$n$--subsets of $C(G,\T)$ whose elements are functions $\phi$ that map some point of $I$ into
$\phi(a)\cdot\mathcal V_{\varepsilon}$ with probability at least $\varepsilon^n$.

\begin{definition}\label{a:def}
Let $G$ be a locally compact Abelian group, $I \subseteq G$ and $a\in
G$. For $\Delta\subseteq C(G,\T)$,  $n, \ell \in \N$, and  $\varepsilon>0$,
we define $\mathcal{A}_{a,\Delta,n,\ell,\varepsilon,I}$ to be the set of all $\ell$--subsets $\Lambda\subseteq I$ such that, for every $F\subseteq\Delta$ with $F\in\mathfrak{P}_{I,a,n,\varepsilon}$,
\[
\Lambda \cap N^{\vartriangleleft}\left(F,\varepsilon,a\right)\neq \emptyset.
\]
\end{definition}
The set $\mathcal{A}_{a,\Delta,n,\ell,\varepsilon,I}$ is made of all
the $\ell$--subsets of $I$ containing some point which is mapped
by all functions $\phi\in F\subseteq\Delta$ into $\phi(a)\cdot\mathcal{V}_{\varepsilon}$.
In what follows we regard the sets
$\mathcal{A}_{a,\Delta,n,\ell,\varepsilon,I}$ as events in the probability
space determined by the restriction of $\lambda_{\g^\ell} $ to $I^\ell$,
see the remarks at the end
of Section \ref{notationandterminology}.

The sets $\mathfrak{P}_{I,a,n,\varepsilon}$ and
$\mathcal{A}_{a,\Delta,n,\ell,\varepsilon,I}$ are tailored to facilitate
these estimates, as the following lemma shows.
\begin{lemma}\label{aest}
Let $I\subseteq G$  be a subset of positive Haar measure
and let $\Delta\subseteq C(G,\T)$ be an $N$--subset. Consider as well
$a\in G$, $\vep>0$ and $n\in \N$.
Then
\begin{align*}
\P\left(\mathcal{A}_{a,\Delta,n,\ell,\varepsilon,I}^{\mathbf{c}}\right)&\leq
\binom{N}{n}\left(1-\vep^n\right)^\ell\leq \left(\frac{Ne}{n}\right)^
n\left(1-\vep^n\right)^\ell
\end{align*}
\end{lemma}
\begin{proof}
The second inequality is a well--known estimate of binomial
coefficients. For the first inequality, we observe  that \[
\mathcal{A}_{a,\Delta,n,\ell,\varepsilon,I}^{\mathbf{c}}=\bigcup\Big\{\big[\left(\mfteiba
\cap I\right)^{\mathbf{c}}\big]^\ell:F\subseteq \Delta,
F\in \mathfrak{P}_{I,a,n,\varepsilon}\Big\},\]
where $\mathbb{P}\Big(\big[\left(\mfteiba \cap
I\right)^{\mathbf{c}}\big]^\ell\Big)\leq (1-\varepsilon^n)^\ell$
since $\lambda_{\g}  \bigl(\left(\mfteiba \cap
I\right)^{\mathbf{c}}\bigr)\leq \lambda_\g(I)(1-\varepsilon^n)$
for every $F\in \mathfrak{P}_{I,a,n,\varepsilon}$. The inequality
follows because there are $\binom{N}{n}$ $n$--subsets of $\Delta$.
\end{proof}

In the following definition we introduce the events
$\mathcal{B}_{a,\Delta^\ast,n,\ell^\ast,\varepsilon,I^\ast}$ which
correspond to the determining sets we are constructing.
\begin{definition}\label{def:bs}
Let $G$ be a LCA group and let $I^\ast=(I_k)_{k\in \N}$ and
$\Delta^\ast=(\Delta_k)_{k\in \N}$ be sequences of
subsets of $G$ and $C(G,\T)$, respectively.
Let as well a sequence $\ell^\ast=(\ell_k)_{k\in \N}$ of positive integers, $a\in G$, $n\in \N$ and $\varepsilon>0$ be given. We define the event
       \[\mathcal{B}_{a,\Delta^\ast,n,\ell^\ast,\varepsilon,I^\ast}=\left\{
       (\Lambda_k)_{k\in \N}\colon \mbox{There exists  $N\in \N$ such that }
       \Lambda_k\in \mathcal{A}_{a,\Delta_k,n,\ell_k,\varepsilon,I_k}
       \mbox{ for } k\geq N\right\}
.\]
\end{definition}
We now estimate the probability of the events
$\mathcal{B}_{a,\Delta^\ast,n,\ell^\ast,\varepsilon,I^\ast}$ regarded as
events in the probability space $\prod_k I_k^{\ell_k}$, where each factor
is assumed to carry
the probability measure induced by the restriction of Haar
measure on $G^{\ell_k}$. It then follows that given
a sequence of intervals $(I_k)_{k\in \N}$ and a sequence
$(\Delta_k)_{k\in \N}$ of finite subsets of $C(G,\T)$, a randomly
chosen sequence of $\ell_k$-subsets of the $I_k$'s belongs to
$\mathcal{B}_{a,\Delta^\ast,n,\ell^\ast,\varepsilon,I^\ast}$ with
probability one as long as the growth of $\ell_k$ is  large enough
(where enough is controlled by $\log k$ and the cardinality of the
$\Delta_k$'s).
\begin{lemma}
\label{firtsglobal}
Consider the  sequences
\begin{enumerate}
\item $I^\ast=(I_k)_{k\in \N}$,
with $I_k\subseteq G$ of nonzero Haar measure,
\item $\Delta^\ast=(\Delta_k)_{k\in \N}$, with
$\Delta_k\subseteq C(G,\T)$ and $L_k:=|\Delta_k|<\infty$, and
\item $\ell^\ast=(\ell_k)_{k\in \N}\subseteq \N$.
\end{enumerate}
Let $n\in \N$ and $\varepsilon>0$. If  there are $\gamma >0$ and $k_0\in \N$ such that
\begin{equation}\label{condition1}
\frac{-n\log L_k}{\varepsilon^n}+\ell_k>\frac{-(1+\gamma)\log k}{\log(1-\varepsilon^n)}
\end{equation}
for every $k\geq k_0$, then
$\P\left(\mathcal{B}_{a,\Delta^\ast,n,\ell^\ast,\varepsilon,I^\ast}\right)=1$
for every $a\in G$.
\end{lemma}
\begin{proof}
Since all indices except $k$ are fixed throughout the proof
we denote
by $\mathcal{A}_k$
the set $\mathcal{A}_{a,\Delta_{k}, n,\ell_{k},\varepsilon,I_{k}}$
and
for $k'\in \N$ we identify the set $\mathcal{A}_{k'}$ with the subset
$\prod_k \mathcal{X}_k \subseteq \prod_k I_k^{\ell_k}$
defined by $\mathcal{X}_{k'}:=\mathcal{A}_{k'}$ and
$\mathcal{X}_k:=I_k^ {\ell_k}$ for $k\neq k'$.
Note that the probabilities of the event $\mathcal{A}_{k'}$ in the
probability space  $\prod_{k} I_k^{\ell_k}$ and
in $I_{k'}^{\ell_{k'}}$ coincide.

Using this identification we have that
\[\mathcal{B}_{a,\Delta^\ast,n,\ell^\ast,\varepsilon,I^\ast}
=\bigcup_{N\in \N}\bigcap_{k\geq N} \mathcal{A}_{k}=\limsup \mathcal{A}_k.\]
We now see that
$\P(\mathcal{B}_{a,\Delta^\ast,n,\ell^\ast,\varepsilon,I^\ast}^{\mathbf{c}})=0$.
In fact, it follows from Lemma \ref{aest} that
\[
\sum_{k\ge k_0}\P\left(\mathcal{A}_k
^{\mathbf{c}}\right)
\leq \sum_{k\ge k_0} \left(\frac{L_k e}{n}\right)^n (1-\varepsilon^n)^{\ell_k}
=\left(\frac{e}{n}\right)^n\sum_{k\ge k_0} (1-\varepsilon^n)^{\frac{n
\log L_k}{\log(1-\varepsilon^n)}+\ell_k}.  \] Since
$\log(1-\varepsilon^n)<-\varepsilon^n$, we obtain
$\tfrac{n\log L_k}{\log(1-\varepsilon^n)}+\ell_k>\tfrac{-n\log
L_k}{\varepsilon^n}+\ell_k$, and by (\ref{condition1}) we get
\[
\sum_{k\ge k_0}\P\left(
\mathcal{A}_k^{\mathbf{c}}\right)
\leq \left(\frac{e}{n}\right)^n \sum_{k\ge
k_0} (1-\varepsilon^n)^{\frac{-n\log
L_k}{\varepsilon^n}+\ell_k} \leq \sum_{k\ge k_0}
\left[(1-\varepsilon^n)^{\frac{-(1+\gamma)}{\log(1-\varepsilon^n)}}\right]^{\log
k}, \] which is a convergent series of the form $\sum_{k\geq k_0}
x^{\log k}$ with $|x|<\tfrac 1 e$.
The Borel--Cantelli Lemma then implies
\[\P(\mathcal{B}_{a,\Delta^\ast,n,\ell^\ast,\varepsilon,I^\ast}^{\mathbf{c}})=
\P\left(\left(\limsup\mathcal{A}_k\right)^{\mathbf{c}}\right)=0,\]
as required.
\end{proof}
\section{Matching intervals  and characters}\label{sec:match}
In this section  $G=\R$ or $G=\Z$. Recall that for  $p\in \R[x]$,
$\psi_p(t)=e^{2\pi i p(t)}$, and for $\tau\in \R$, $\chi_\tau(x)=e^{2\pi
i \tau x}$.
\begin{definition}{\rm (\cite[Definition 2.74]{drmotichy97})}\label{discrep}
Let $\mathbf{x}\colon [0,+\infty)\to \R^n$ be a continuous function.
The \emph{continuous discrepancy of $ \mathbf{x}$ in $[0,T]$} is defined by
\[
D_T (\mathbf{x})= \sup_{\mathcal{V}
\subseteq \T^n }\left|\frac{1}{T}\int_0^T
\mathbf{1}_{\mbox{\tiny$\mathcal{V}$}}\left(e^{2\pi i
\mathbf{x}(t)}\right)\, dt-\lambda_{\T^n}(\mathcal{V})\right|,
\]
where  $\mathbf{1}_{\mbox{\tiny$\mathcal{V}$}}$ denotes the characteristic
function of the set $\mathcal{V}$,  $e^{2\pi i \mathbf{x}(t)}$ stands for
the vector $\left(e^{2\pi i x_1(t)},\ldots,e^{2\pi i x_n(t)}\right)\in
\T^n$  and the supremum is taken over all rectangles in $\T^n$ with sides
parallel to the axes.
\end{definition}
\begin{definition}{\rm (\cite[Definition 2.70 and Theorem 2.75]{drmotichy97})}\label{def:cwd}
A function $\mathbf{x}:[0,+\infty) \to\R^n
$ is \emph{continuously well distributed modulo 1} if it is continuous and $\displaystyle{\lim_{T\to\infty }
D_T (\mathbf{x}(t+\tau))=0}$ uniformly in $\tau$.
\end{definition}
We propose the following definition to make our notation lighter.
\begin{definition}\label{def:pli}
The polynomials $\{p_1,\ldots,p_n\}\subseteq \R[x]$ are \emph{strongly
linearly independent over $\Q$} if for each nonzero $(h_1,\ldots,h_n)\in
\Z^n$ the polynomial $\sum_{j=1}^n h_jp_j$ is nonconstant.

The functions $\{\psi_{p_1},\ldots,\psi_{p_n}\}$ are \emph{strongly
linearly independent over $\Q$} if $\{p_1,\ldots,p_n\}\subseteq
\R[x]$ are strongly linearly independent over $\Q$.
\end{definition}

Let $H$ be a Hamel basis of $\R$ over $ \Q$ and let $C_H$ denote
the constant polynomials with values in $H$. The polynomials
$\{p_1,\ldots,p_n\}\subseteq \R[x]$ are  strongly linearly
independent over $\Q$ if and only if the set $\{p_1,\ldots,p_n\}
\cup C_H$ is  linearly independent over $\Q$.
\begin{theorem}{\rm
(\cite[Corollary of Theorems 2.73 and 2.79]{drmotichy97})}\label{wdpoly}
If $\{p_1,\ldots,p_n\}\subseteq \R[x]$ are strongly linearly independent
over $\Q$, then the function $\mathbf{x}\colon [0,+\infty)\to \R^n$
defined by $\mathbf{x}(t)=\left(p_1(t),\ldots,p_n(t)\right)$ is
continuously well distributed modulo 1.
\end{theorem}
\begin{proof}
By Weyl's criterion for continuous well--distribution \cite[Theorem
2.73]{drmotichy97} it suffices to prove that for every nonzero
$\mathbf{h}=(h_1,\ldots,h_n)\in \Z^n$,
\begin{equation}
\label{weylwd}
\lim_{T\to \infty}\frac{1}{T}\int_0^T e^{2\pi i \mathbf{h}\cdot x(t+\tau)}\,dt=0,
\end{equation}
uniformly in $\tau\geq 0$.

Let $\mathbf{h}=(h_1,\ldots,h_n)\in \Z^n$ be nonzero. Since the polynomial
$q_{h}(t)=\sum_{j=1}^n h_j p_j(t)$ is nonconstant, there exist $t_0\in
\R$ and $C>0$ such that $|q_h(t)|\geq C$ and $q_h^{\prime\prime}(t)$
has constant sign for every $t\geq t_0$. From \cite[Theorem
2.79]{drmotichy97} it then follows that $q_h$ is
continuously well distributed, and Weyl's criterion applied to $q_h$ shows
that \eqref{weylwd} holds.
\end{proof}
If $F$ is a family of $\T$-valued functions, we are trying to estimate
how long an interval $I$  should be for $F$ to  be as  likely as expected
to send some element of $I$ into a  fixed  neighbourhood of $\T$.
We next see, as a consequence of Theorem \ref{wdpoly}, that  for
polynomial phase functions
generated by  strongly linearly
independent polynomials,  this happens as soon as the length of $I$
exceeds a bound that depends only on the cardinality of the family
and the size of the neighbourhood.

\begin{theorem}\label{li}
Let $F=\{\psi_{p_1},\ldots,\psi_{p_n}\}$ where
$\{p_1,\ldots,p_n\}\subseteq \R[x]$ are strongly linearly independent over
$\Q$, and let $\gamma>0$. Then there exists $L(F,\gamma)>0$ such that
for every interval $I\subseteq \R$ with $\lambda_\R(I)\geq L(F,\gamma)$
and every $a\in \R$,
\[
\lambda_\R\left(N^{\vartriangleleft}\left(F,\delta,a\right)\cap I\right)
\geq  \left(  (2\delta)^n- \gamma\right)\lambda_\R(I),
\]
for every $\delta>0$ with $(2\delta)^n-\gamma>0$.
In particular, $F\in \mathfrak{P}_{I,n,\delta}$ if
$\lambda_\R(I)\geq L(F,\delta^n (2^n-1))$.
\end{theorem}
\begin{proof}
Fix $a\in \R$ and $\gamma>0$. Define  $\mathbf{x}_{\mbox{\tiny
$F$}}\colon [0,+\infty)\to \R^n$ by ${\displaystyle
\mathbf{x}_F(t)=\left(p_1(t),\ldots,p_n(t)\right)}$ and put, for each
$\delta>0$,  ${\displaystyle  \mathcal{V}_{\delta,a}=\psi_{p_1}(a)\cdot
\mathcal{V}_\delta\times \cdots\times \psi_{p_n}(a)\cdot
\mathcal{V}_\delta}\subseteq \T^n$. By Theorem \ref{wdpoly} the function
$\mathbf{x}_{\mbox{\tiny $F$}}(t) $ is continuously well distributed.
There is accordingly $L(F,\gamma)>0$ such that  $T\geq L(F,\gamma)$
implies $D_T(\mathbf{x}(t+\tau))\leq \gamma$ for every $\tau$, i.e.
\begin{equation}\label{eq:wd}
\frac{1}{T}\int_0^T
\mathbf{1}_{\mbox{\tiny$\mathcal{V}_{\delta,a}$}}
\left(e^{2\pi i \mathbf{x}_F(t+\tau)}\right)\, dt\geq (2\delta)^n-\gamma.
\end{equation}
Let now $I=[\tau_0,\tau_0+L]\subseteq \R$ be an arbitrary interval of
length  $L\geq L(F,\gamma)$. Taking into account the definition of $
\mathcal{V}_{\delta,a}$, inequality \eqref{eq:wd} applied to $\tau=\tau_0$
and $T=L$ implies that
\begin{align*}
L\cdot \left((2\delta)^n-\gamma\right)&\leq
\lambda_\R\left(\left\{t\in[0,L] \colon \psi_{p_j}(t+\tau_0)\in
\psi_{p_j}(a)\cdot \mathcal{V}_\delta\mbox{ for
} j=1,\ldots n
\right\}\right)\\
&= \lambda_\R\left(N^{\vartriangleleft}\left(F,\delta,a\right)\cap I\right).
\end{align*}
\end{proof}
The same argument of Theorem \ref{li} with well distributed sequences
instead of continuously well distributed functions
can be used for $G=\Z$.
\begin{corollary}\label{cor:li}
Let $F=\{\psi_{p_1},\ldots,\psi_{p_n}\}$ where
$\{p_1,\ldots,p_n\}\subseteq \R[x]$ have coefficients in $[0,1)$ and
are strongly linearly independent over $\Q$, and let $\gamma>0$. Then
there is $L(F,\gamma)>0$ such that, for every interval $I\subseteq \Z$
with $|I|\geq L(F,\gamma)$ and every $a\in \Z$,
\[|N^{\vartriangleleft}\left(F,\delta,a\right)\cap I|\geq
\left((2\delta)^n- \gamma\right)|I|,\]
for every $\delta>0$ with $(2\delta)^n-\gamma>0$. In particular, $F\in
\mathfrak{P}_{I,n,\delta}$ if $|I|\geq L(F,\delta^n (2^n-1))$.
\end{corollary}

To close this section we consider sets of characters instead of
sets of more general continuous $\T$-valued functions. In this
case $F\in \mathfrak{P}_{I,a,n,\varepsilon}$ if and only if $F\in
\mathfrak{P}_{I,n,\varepsilon}$, and therefore we localize our arguments
at the identity. For some special sets of characters $F$  we can actually
find a more concrete bound for Theorem \ref{li}.
\begin{theorem}\label{mfetnei1}
Let  $0<\frac{q_1}{p_1}<\cdots<\frac{q_n}{p_n}=1$ be a finite sequence of
rationals and $\varepsilon\in \left(0,\frac 1 2\right)$. If $I\subseteq
\R$ is an interval with $\lambda_\R(I)\geq p_1\cdots p_{n-1}$,
then $F:=\{\chi_{\frac{q_1}{p_1}},\dots,\chi_{\frac{q_n}{p_n}}\}\in
\mathfrak{P}_{I,n,\varepsilon}$.
\end{theorem}
\begin{proof}
Let $N:=p_1\cdots p_{n-1}$. We can assume that the fractions
$\frac{q_1}{p_1},\ldots, \frac{q_{n-1}}{p_{n-1}}$ are irreducible. If they
are not  we work with the simplified fractions and obtain a smaller $N$
that works.

We first assume $\lambda_\R(I)=N$ and define $J_k:=\left\{0,1,\ldots,
\floor{p_k\varepsilon}, p_k-\floor{p_k\varepsilon},\ldots,p_k-1\right\}$
for each $k=1,\dots,n-1$. We claim that
\begin{enumerate}
\item for each $\vec \jmath \in \prod_{k=1}^{n-1} J_k$ there exists
$z_{\vec \jmath}\in\Z$ such that either $[z_{\vec \jmath},z_{\vec
\jmath}+\varepsilon]$ or $[z_{\vec \jmath}-\varepsilon,z_{\vec \jmath}]$
is contained in $\mfteib$,
\item for $z_{\vec 0}$ we have $[z_{\vec 0}-\varepsilon,z_{\vec
0}+\varepsilon]\subseteq \mfteib$, and
\item the integers $z_{\vec \jmath}$ are all different.
\end{enumerate}
Indeed, for fixed $j_{n-1}\in J_{n-1}$ we consider the set
$L_{n-1,j_{n-1}}$ consisting of those integers in $I$ which are mapped
to $e^{2\pi i j_{n-1}/p_{n-1}}$ by $\chi_{\frac{q_{n-1}}{p_{n-1}}}$, i.e.
\[L_{n-1,j_{n-1}}=\left\{z\in \Z\cap I\colon \frac{q_{n-1}}{p_{n-1}}\cdot
z=\dfrac{j_{n-1}}{p_{n-1}}+\ell \mbox{ for some $\ell\in \Z$ }\right\}.\]
Since (the class of) $q_{n-1}$ is a generator of the cyclic group
$\Z/p_{n-1}\Z$, the set $L_{n-1,j_{n-1}}$ contains precisely $N/p_{n-1}$
integers with a distance of $p_{n-1}$ between consecutive ones. For $z\in
L_{n-1,j_{n-1}}$, either $[z,z+\varepsilon]$ or $[z-\varepsilon,z]$
is contained in $I$ and  sent into $\mathcal{V}_\varepsilon$
by both $\chi_1$ and $\chi_{\frac{q_{n-1}}{p_{n-1}}}$. For $z\in
L_{n-1,0}$, both characters map $[z-\varepsilon,z+\varepsilon] $ into
$\mathcal{V}_\varepsilon$.

Next we fix $j_{n-2}\in J_{n-2}$ and consider the set $L_{n-2,j_{n-2}}$
of those elements of $L_{n-1,j_{n-1}}$ which are sent to $e^{2\pi i
j_{n-2}/p_{n-2}}$ by $\chi_{\frac{q_{n-2}}{p_{n-2}}}$, i.e.
\[L_{n-2,j_{n-2}}=\left\{z\in L_{n-1,j_{n-1}}\colon \frac{q_{n-2}}{p_{n-2}}\cdot
z=\dfrac{j_{n-2}}{p_{n-2}}+\ell \mbox{ for some $\ell\in \Z$ }\right\}.\]
As before, exactly $N/(p_{n-1}p_{n-2})$ elements of $L_{n-1,j_{n-1}}$
belong to $L_{n-2,j_{n-2}}$, and the distance between any two
consecutives is $p_{n-1}p_{n-2}$. For $z\in L_{n-2,j_{n-2}}$,
either $[z,z+\varepsilon]$ or $[z-\varepsilon,z]$ is sent into
$\mathcal{V}_\varepsilon$ by $\chi_1$, $\chi_{\frac{q_{n-1}}{p_{n-1}}}$
and $\chi_{\frac{q_{n-2}}{p_{n-2}}}$. For $z\in L_{n-2,0}$,
these characters map $[z-\varepsilon,z+\varepsilon]$ into
$\mathcal{V}_\varepsilon$.

After $(n-1)$ steps the components of $\vec \jmath \in \prod_{k=1}^{n-1}
J_k$ have been fixed, the set $L_{1,j_1}$ contains precisely one
integer, say $z_{\vec \jmath}\in \bigcap_{k=1}^{n-1} L_{k,j_k}$,
and $[z_{\vec \jmath},z_{\vec \jmath}+\varepsilon]$ or $[z_{\vec
\jmath}-\varepsilon,z_{\vec \jmath}]$ is contained in $\mfteib$. Since
each $\vec \jmath$ produces a different $z_{\vec \jmath}$, our claim
is proved.

Since $|J_k|=2 \floor{p_k\varepsilon}+1$ for each $k$, and the interval
around $z_{\vec \jmath}$ have length at least $\varepsilon$, the intervals
constructed in the previous claim have a total length of
\begin{equation}\label{eqint}
\varepsilon\cdot\left[\prod_{k=1}^{n-1}(2\floor{p_k\varepsilon}+1)+1\right]\geq
2\varepsilon^n N.
\end{equation}
In fact, since $2\floor{n\varepsilon}+1\geq n\varepsilon$ for
every $n\in \N$, in case $p_i\varepsilon,p_j\varepsilon\geq
\frac{1}{2-\sqrt{2}}$ for some $i\neq j$ then
$(2\floor{p_i\varepsilon}+1)(2\floor{p_j\varepsilon}+1)\geq
(2p_i\varepsilon-1)(2p_j\varepsilon-1)\geq
(\sqrt{2}p_i\varepsilon)(\sqrt{2}p_j\varepsilon)=2\varepsilon^2p_ip_j$,
and (\ref{eqint}) follows. On the other hand, if $1\leq
p_i\varepsilon \leq\frac{1}{2-\sqrt{2}}$ for some $i$, then
$3=(2\floor{p_i\varepsilon}+1)\geq 2 p_i\varepsilon$ and
inequality \eqref{eqint} holds. The only remaining case is when
$p_i\varepsilon<1$ for every $i$ except at most one $i_0$. In that case
from $2\floor{p_{i_0}\varepsilon}+1\geq 2p_{i_0}\varepsilon-1$ we also
obtain that (\ref{eqint}) holds since its left hand side is bounded
below by
\[
\varepsilon\cdot[(2\floor{p_{i_0}\varepsilon}+1)+1]\geq
\varepsilon\cdot(p_1\varepsilon) \cdots (2 p_{i_0})\cdots
(p_{n-1}\varepsilon)=2\varepsilon^n N.
\]

We have thus shown that $\lambda_\R\left( \mfteib\right)\geq
2\varepsilon^n\lambda_\R(I)$ when $\lambda_\R(I)=N$, as desired. In
case $\lambda_\R(I)>N$, there exist $j\in \N$ and $\delta\in[0,N)$
such that $\lambda_\R(I)=jN+\delta$. Therefore, the interval $I$ can
be split into $j$--many subintervals of length $N$ and another one of
length $\delta$. In each of the intervals of length $N$ we can argue as
above  and find a family of  subintervals of $\mfteib$ whose accumulated
length is $2N \varepsilon$. We then obtain that
\[
\lambda_\R\left(\mfteib\right)\geq 2jN\varepsilon^n\geq
2\varepsilon^n\left(1-\dfrac{1}{j+1}\right)\left(jN+\delta\right)\geq
\varepsilon^n \lambda_\R(I).
\]
\end{proof}
\begin{corollary}\label{mfetnei}
Let $0<\frac{q_1}{p_1}<\cdots<\frac{q_n}{p_n}$ be a finite sequence of
rationals and $\varepsilon\in \left(0,\frac 1 2\right)$. If $I\subseteq
\R$ is an interval with $\lambda_\R(I)\geq p_1\cdots p_n q_n^{n-2}$,
then $F:=\{\chi_{\frac{q_1}{p_1}},\dots,\chi_{\frac{q_n}{p_n}}\}\in
\mathfrak{P}_{I,n,\varepsilon}$.
\end{corollary}
\begin{proof}
From $\lambda_\R(I)\geq p_1\cdots p_n q_n^{n-2}$ we get
$\lambda_\R(\frac{q_n}{p_n}I)\geq (q_np_1)\cdots (q_{n} p_{n-1})$,
and Theorem \ref{mfetnei1} then implies $\frac{p_n}{q_n}
F:=\{\chi_{\frac{p_nq_1}{q_np_1}},\ldots,
\chi_{\frac{p_nq_{n-1}}{q_np_{n-1}}},\chi_1\}\in
\mathfrak{P}_{\frac{q_n}{p_n} I,n,\varepsilon}$. The result
follows because $N^{\vartriangleleft}(\frac{p_n}{q_n} F,\varepsilon)\cap
\frac{q_n}{p_n}I=\frac{q_n}{p_n}\left(N^{\vartriangleleft}\left(F,\varepsilon\right)\cap
I\right)$ implies $\frac{p_n}{q_n} F\in \mathfrak{P}_{\frac{q_n}{p_n}
I,n,\varepsilon}$ if and only if ${\displaystyle F\in \mathfrak{P}_{
I,n,\varepsilon}}$.
\end{proof}
The interval can be shortened in the presence of certain algebraic relations in $F$.
\begin{corollary}\label{mfetneisimpl} Let
$0<\frac{q_1}{p}<\cdots<\frac{q_n}{p}$ be a finite sequence
of rationals and $\varepsilon\in\left(0,\frac 1 2\right)$. If
$I\subseteq \R$ is an interval with $\lambda_\R(I)\geq p q_n^{n-2}$,
then $F:=\{\chi_{\frac{q_1}{p}},\dots,\chi_{\frac{q_n}{p}}\}\in
\mathfrak{P}_{I,n,\varepsilon}$.
\end{corollary}
\begin{proof}
From $\lambda_\R(I)\geq p q_n^{n-2}$
we obtain $\lambda_\R(\frac{q_n}{p}I)\geq
q_n^{n-1}$ and Theorem \ref{mfetnei1} then asserts
$\frac{p}{q_n}F=\{\chi_{\frac{q_1}{q_n}},\chi_{\frac{q_2}{q_n}},\dots,\chi_1\}\in
\mathfrak{P}_{\frac{q_n}{p}I,n,\varepsilon}$, i.e. $F\in
\mathfrak{P}_{I,n,\varepsilon}$.
\end{proof}
A considerably shorter interval is needed when $F$ is sparse enough.
\begin{corollary}\label{mftneisparse}
Let $F:=\{\chi_{\tau_1},\dots,\chi_{\tau_n}\}$ be such that
$\frac{\tau_{j+1}}{\tau_j}>\frac{1}{2\varepsilon}$, $j=1,\ldots,n-1$,
for some $\varepsilon\in\left(0,\frac 1 2\right)$. If $I\subseteq \R$
is an interval with $\lambda_\R(I)\geq \frac{1}{\tau_1}$, then $F\in
\mathfrak{P}_{I,n, \varepsilon}$.
\end{corollary}
\begin{proof}
If $\lambda_\R(I)=\frac{1}{\tau_1}$, there exists $I_1\subseteq
I$ such that $\lambda_\R(I_1)=\frac{2\varepsilon}{\tau_1}$
and $\chi_{\tau_1}[I_1]=\mathcal{V}_\varepsilon$. Then
$\chi_{\tau_2}[I_1]=\T$ and there exists $I_2\subseteq I_1$ such that
$\lambda_\R(I_2)=2\varepsilon\cdot \frac{2\varepsilon}{\tau_1}$
and $\chi_{\tau_2}[I_2]=\mathcal{V}_\varepsilon$. At
the $n$th step we find $I_n\subseteq I_{n-1}$ such
that $\lambda_\R(I_n)=\frac{(2\varepsilon)^n}{\tau_1}$ and
$\chi_{\tau_n}[I_n]\subseteq \mathcal{V}_\varepsilon$. It follows that
$I_n\subseteq\mfteib$ and $F\in \mathfrak{P}_{I,n,\varepsilon}$. Finally,
for $\lambda_\R(I)\geq \frac{1}{\tau_1}$ there exist
$j\in \N$ and $\delta\in [0,\frac{1}{\tau_1})$ such that
$\lambda_\R(I)=\frac{j}{\tau_1}+\delta$. Split $I$ in $j$-many
subintervals of length $\frac{1}{\tau_1}$ plus another one of length
$\delta$. By the above argument each of the former intervals contains
a subinterval of length $\frac{(2\varepsilon)^n}{\tau_1}$ and the proof
then goes as in Theorem \ref{mfetnei1}.
\end{proof}

\section{Random Bohr dense subsets}\label{sec:RBDS}

In this section we combine the results of Sections
\ref{sec:appdense} and \ref{sec:match} in order to show that if enough points
are randomly chosen from each element of a sequence of sufficiently large intervals, then
almost surely we obtain a Bohr--dense set.
The  estimates in Section \ref{sec:match} are first used to find criteria
for a collection of finite choices in a sequence of long enough intervals
of real numbers to be an $(F,\R,\varepsilon)$--matching set for every
$m$--set $F$ (with $m\in \N$
fixed) of strongly linearly independent polynomial phase functions.
The estimates of Section \ref{sec:appdense} are then used to see
that these criteria are almost surely met. This  yields sets that
are $(F,\R,\varepsilon)$--matching for every finite set of degree
$n$ polynomial phase functions and every $\varepsilon>0$, that is,
sets of uniqueness for $\ap_{\mathfrak{C}_{n}}$.

\begin{lemma}\label{lem:BD-R}
Let $I^\ast=(I_k)_{k\in \N}$ be a sequence of intervals
$I_k=[a_k,a_k+b_k]\subseteq \R$ with $\limsup(b_k)_{k\in \N}=+\infty$. For
each $k\in \N$ put $t_k:=\max(|a_k|,|a_k+b_k|)$
and let $\Delta_k$ be a finite subset of $C(\R,\T)$
whose restrictions to $[-t_k,t_k]$ are $\varepsilon_k$--dense in the
restriction of ${\mathfrak C}_n$ to $[-t_k,t_k]$, where
$\varepsilon_k\to 0$ as $k\to \infty$. Let $\Delta^\ast=(\Delta_k)_{k\in
\N} $, and let $\ell^\ast=(\ell_k)_{k\in \N}$ be a sequence of positive
integers.
If $(d_s)_{s\in \N}\subseteq \R$ is a dense subset and for some $m\in
\Z$ and $\varepsilon>0$ we have $(\Lambda_k)_{k\in \N}\in \bigcap_{s\in
\N} \mathcal{B}_{d_s,\Delta^\ast,m,\ell^\ast,\varepsilon,I^\ast}$,
then $\Lambda=\bigcup_{k\in \N} \Lambda_k $ is an
$(F,\R,\varepsilon)$--matching
set for every $m$--subset $F$ of $\mathfrak{C}_n$ strongly linearly independent over $\Q$.
\end{lemma}
\begin{proof}
Let $F=\{\psi_{p_1},\ldots \psi_{p_m}\}\subseteq \mathfrak{C}_n$ with
$\{p_1,\ldots,p_m\}\subseteq \R[x]$ strongly linearly independent over
$\Q$. Fix $x_0\in \R$ and consider  $N\in \N$ and  $s_0 \in \N$ with
$\frac 3 N<\varepsilon$ and
\begin{equation}\label{dense1}
d_a \bigl(\psi_{p_j}(d_{s_0}),\psi_{p_j}(x_0 )\bigr)<\frac{1}{N},\mbox{ for } j=1,\dots,m.
\end{equation}
Since $(\Lambda_k)_k\in \mathcal{B}_{d_{s_0},\Delta^\ast,
m,\ell^\ast,\frac{1}{N},I^\ast}$, there is $k_0$ such that for every $k\geq k_0$,
\begin{equation}\label{eq:Lamb1}
\Lambda_k\in \mathcal{A}_{d_{s_0},\Delta_k,m,\ell_k,\frac{1}{N},I_k}.
\end{equation}
Applying Theorem \ref{li} to $\delta,\gamma>0$ such that
$\delta<\frac{1}{N}<2\delta$ and $(2\delta)^m-\gamma>N^{-m}$, we
can find $ k\geq k_0$ large enough to satisfy (\ref{eq:Lamb1}),
$\varepsilon_k<\frac 1 2(\frac 1 N-\delta)$, $b_k>L(F,\gamma)$ and
$d_{s_0}\in [-t_k,t_k]$. Since $\Delta_k$ is $\varepsilon_k$--dense in
the restriction of ${\mathfrak C}_n$ to $[-t_k,t_k]$, for every
$j=1,\dots,m$ there exists $\phi_j \in \Delta_k $ such that for every
$t\in [-t_k,t_k]$,
\begin{equation}\label{ddense1}
d_a\left(\phi_j(t),\psi_{p_j}(t)\right)<
\varepsilon_k<\frac{1}{2}\left(\frac{1}{N}-\delta\right).
\end{equation}
From (\ref{dense1}) and (\ref{ddense1}) it follows that, for $t\in
N^{\vartriangleleft}\left(F,\delta,d_{s_0}\right)\cap I_k$, we have
\[
d_a\left(\phi_j(t),\phi_j(d_{s_0})\right)\le d_{a}\left(
\phi_j(t),\psi_{p_j}(t)\right)+
d_a\left(\psi_{p_j}(t),\psi_{p_j}(d_{s_0})\right)+
d_a\left(\psi_{p_j}(d_{s_0}),\phi_j(d_{s_0})\right)<\frac{1}{N},
\]
thus $N^{\vartriangleleft}\left(F,\delta,d_{s_0}\right)\cap I_k\subseteq
N^{\vartriangleleft}(\tilde{F},\frac{1}{N},d_{s_0})\cap I_k$,
where $\tilde{F}:=\{\phi_1,\dots,\phi_m\}$. Since $b_k>L(F,\gamma)$
implies $\lambda_\R(N^{\vartriangleleft}(F,\delta,d_{s_0})\cap
I_k)\geq N^{-m}\lambda_\R(I_k)$, we conclude that
$\tilde{F}\in \mathfrak{P}_{I_k,d_{s_0},m,\frac{1}{N}}$,
which together with \eqref{eq:Lamb1} implies $\Lambda_k\cap
N^{\vartriangleleft}(\tilde{F},\frac{1}{N},d_{s_0})\neq \varnothing$. From
\eqref{dense1} and \eqref{ddense1} we then obtain that for $y\in
\Lambda_k\cap N^{\vartriangleleft}(\tilde{F},\frac{1}{N},d_{s_0})$,
\begin{align*}
d_a\left( \psi_{p_j}(y),\psi_{p_j}(x_0)\right)&\leq
d_a\left( \psi_{p_j}(y),\phi_j(y)\right)+d_a\left(
\phi_j(y),\phi_j(d_{s_0})\right)\\
&+ d_a\left( \phi_j(d_{s_0}), \psi_{p_j}(d_{s_0})\right)+
d_a\left(\psi_{p_j}(d_{s_0}),\psi_{p_j}(x_0)\right)<\frac 3 N,
\end{align*}
i.e. $\psi_{p_j}(y)\in \psi_{p_j}(x_0)\cdot
\mathcal{V}_{\frac{3}{N}}\subseteq \psi_{p_j}(x_0)\cdot
\mathcal{V}_{\varepsilon}$. In conclusion, $\Lambda$ is an
$(F,\R,\varepsilon)$--matching set.
\end{proof}
\begin{theorem}
\label{BD-R}
Let $I^\ast=(I_k)_{k\in \N}$ be a sequence of intervals
$I_k=[a_k,a_k+b_k]\subseteq \R$ with $\limsup(b_k)_{k\in \N}=+\infty$
and $t_k:=\max(|a_k|,|a_k+b_k|)\geq
k$ for every $k\in \N$. For each $k\in \N$ let $\Lambda_k\subseteq I_k$
be a random subset with $|\Lambda_k|=\ell_k$ and $\ell_k\neq O(\log
t_k)$. Then, for any fixed $n\in \N$, the set $\Lambda=\bigcup_{k\in
\N} \Lambda_k$ is, almost surely, a set of uniqueness for
$\ap_{\mathfrak{C}_{n}}$.
\end{theorem}
\begin{proof}
We divide this proof into Steps.
We first determine  that a certain number
of conditions in the selection of the $\Lambda_k$'s are satisfied with
probability one and then we show that, when the sets $\Lambda_k$ meet these
conditions, then the set
$\Lambda=\bigcup_{k\in \N} \Lambda_k$ is a set of uniqueness for $\ap_{\mathfrak{C}_{n}}$.

\textbf{Step 1}: \emph{Exhibiting an event $\mathcal{B}$ of probability one.}\\
Fix $n\in \N$. Since $\ell_k\neq O\bigl(\log t_k\bigr)$, the set
\[
\textstyle{\mathcal{K}_{N,m}=\left\{k\in \N\colon\ell_k\geq (n+1)mN^m\log 2+ \left[mN^{m}\frac{(n+
  4)(n+1)}{2}-\frac{2}{\log\left(1-N^{-m}\right)}
  \right]\log t_k\right\}}
\]
is infinite for each $N,m\in \N$. For each $k\in \N$ define
\[\widetilde{\Delta_k}=\left\{\sum_{r=0}^n
\frac{j_r}{t_k^{r+1}}x^{r}\colon -\floor{t_k^{r+2}}\leq j_r<
\floor{{t_k}^{r+2}},j_r\in\Z,r=0,\dots,n \right\} \subseteq \R_n[x],\]
and $\Delta_k=\left\{\psi_{p}\colon p\in \widetilde{\Delta_k}\right\}$. Observe that $|\Delta_k|=\prod_{r=0}^n(2\floor{t_k^{(r+2)}})\leq 2^{n+1}
t_k^{\frac{(n+1)(n+4)}{2}}$. Fix  $m, N\in \N$.  If $k\in  \mathcal{K}_{N,m}$, then
\begin{align*}
\ell_k&\geq  (n+1)mN^m\log 2+ \left[mN^{m}\frac{(n+
  4)(n+1)}{2}-\frac{2}{\log\left(1-N^{-m}\right)}
  \right]\log t_k
  \\
&\geq m N^m \log |\Delta_k|-\frac{2}{\log\left(1-N^{-m}\right)}\log k,
\end{align*}
and Lemma \ref{firtsglobal} applied to the sequences
$I^{\ast}_{N,m}=(I_k)_{k \in \mathcal{K}_{N,m}}$,
$\Delta^\ast_{N,m}=(\Delta_k)_{k \in \mathcal{K}_{N,m}}$
and $\ell^\ast_{N,m}=(\ell_k)_{k \in \mathcal{K}_{N,m}}$,  yields then
that for every $q\in\R$ the event $\mathcal{B}_{q,\Delta^\ast_{N,m},m,
\ell_{N,m}^\ast,\frac{1}{N},I_{N,m}^\ast}$ occurs with probability
one. Therefore, if $(d_s)_{s\in \N}\subseteq \R$ denotes a countable
dense subset, the event
\[\mathcal B=\bigcap_{s\in \N}\bigcap_{m\in \N}\bigcap_{N  \in \N}
\mathcal{B}_{d_s,\Delta^\ast_{N,m},
m,\ell_{N,m}^\ast,\frac{1}{N},I_{N,m}^\ast},\]
also occurs with probability one.

\textbf{Step  2:} The set \emph{$\Delta_k$ is $\frac{n+1}{t_k}$--dense in the restrictions of
$\mathfrak{C}_n$ to $[-t_k,t_k]$.}\\
Consider any element $f(x)=e^{2\pi ip(x)}$ of $\mathfrak{C}_n$ with
$p(x)=\sum_{r=0}^n a_rx^r\in\R[x]$.
For each $r\in\{0,\dots,n\}$ find $j_r\in\Z$ such that
$\Big|a_r-\frac{j_r}{t_k^{r+1}}\Big|\le\frac{1}{t_k^{r+1}}$.
Then, the polynomial $q(x)=\sum_{r=0}^n\frac{j_r}{t_k^{r+1}}x^r$ is such that
$h(x)=e^{2\pi i q(x)}\in\Delta_k$, and for every $x\in [-t_k,t_k]$ we obtain
\[d_a(h(x),f(x))\le|p(x)-q(x)|\le\sum_{r=0}^n\bigg|a_r-\frac{j_r}{t_k^{r+1}}\bigg|\cdot
|x|^r
\le\sum_{r=0}^n\frac{|x|^r}{t_k^{r+1}}\le\sum_{r=0}^n\frac{t_k^r}{t_k^{r+1}}=
\frac{n+1}{t_k}.\]

\textbf{Step 3}: \emph{If $(\Lambda_k)_{k\in \N}\in \mathcal{B}$, then $
\Lambda=\bigcup_{k \in  \N}\Lambda_k$ is $(F,\R,\varepsilon)$--matching
for  every $F\subseteq \mathfrak{C}_n$ induced by  a finite family of
polynomials that is strongly linearly independent over $\Q$ and every
$\varepsilon>0$.}\\
For such a set $\Lambda$, fix $\varepsilon>0$ and let $F=\{\psi_{p_1},\ldots, \psi_{p_s}\}$
 with $\{p_1,\ldots,p_s\}\subseteq \R_n[x]$. strongly linearly
independent over $\Q$. Fix $N\in \N$ with  $\frac{1}{N}<\varepsilon$. By
our Step 2 above,  Lemma \ref{lem:BD-R} applies  to the sequences
$I^\ast_{N,m}$, $\Delta^\ast_{N,m}$ and $\ell^\ast_{N,m}$ and shows that
$\Lambda$ is an $(F,\R,\varepsilon)$ matching set.

\textbf{Step 4}: \emph{If $(\Lambda_k)_{k\in \N}\in \mathcal{B}$, then
$\Lambda=\bigcup_{k \in  \N}\Lambda_k$ is $(F,\R,\varepsilon)$--matching
for  every finite set $F\subseteq \mathfrak{C}_n$ and every
$\varepsilon>0$.}\\

Fix  $\varepsilon>0$ and let $F=\{\psi_{p_1},\ldots, \psi_{p_s}\}$
with $\{p_1,\ldots,p_s\}\subseteq \R_n[x]$. We may assume that
$\{p_1,\ldots,p_m\}$ are strongly linearly
independent over $\Q$ and that, for each $j=m+1,\ldots, s$, there is a
constant $C_j\in \R$ and there are  integers $Q$ and $Z_{ij}$,
$1\leq i\leq m$ with
\[
p_j= C_j- \sum_{i=1}^{m} \frac{Z_{ij}}{Q}p_i.
\]
Let us define $M=\max\{\sum_{i=1}^m |Z_{ij}|: j\in [m+1,N]\}$ and $\tilde{\varepsilon}=\varepsilon/M$. Fix $N\in \N$ with  $\frac{1}{N}<\tilde{\varepsilon}$.

Since the family  $\tilde{F}=\{p_1/Q,\ldots,p_m/Q\}$ is  strongly
linearly independent over $\Q$, Step 3 above shows that  $\Lambda$
is an $(\tilde{F}, \R,\tilde{\varepsilon})$--matching set. We next see
that $\Lambda$ is also an $(F, \R,\varepsilon)$--matching set.

Let  $a\in \R$. Since $\Lambda$ is an $(\tilde{F},
\R,\tilde{\varepsilon})$--matching set, there is $x_a\in \Lambda_0$
such that
$x_a\in N^{\vartriangleleft}(\tilde{F},\tilde{\varepsilon},a)$,
that is, for each $i=1,\ldots, m$ there are $\delta_i$ with
$|\delta_i|<\tilde{\varepsilon}$ and $M_i\in \Z$ such that
\[ \frac{p_i}{Q}(x_a)-\frac{p_i}{Q}(a)=\delta_i+M_i.\]

Then, for $j=m+1,\ldots, s$,
\[
p_j(x_a)-p_j(a)=\sum_{i=1}^m Z_{ij} \left[\frac{p_i}{Q}(x_a)-
\frac{p_i}{Q}(a)\right]=\sum_{i=1}^m Z_{ij} (\delta_i+M_i).
\]
Since $\sum_{i=1}^{m}Z_{ij}M_i\in \Z$ and
$\left|\sum_{i=1}^{m}Z_{ij}\delta_i\right|\leq \varepsilon$, we see that
$\psi_{p_j}(x_a)\in \psi_{p_j}(a)\cdot \mathcal{V}_\varepsilon$ for every
$j=m+1,\ldots, s$. The same conclusion being obvious for $j=1,\ldots,m$,
it follows that $x_a\in N^{\vartriangleleft}\left(F,\varepsilon,a\right)$,
as we wanted to show.

Having proved that $\Lambda$ is an $(F,\R,\varepsilon)$--matching set for
every $\varepsilon>0$ and every finite set $F\subseteq \mathfrak{C}_{n}$,
an application of Corollary \ref{mat->dense} then concludes the proof.
\end{proof}
In the case of $\ap(\R)$, we also obtain almost sure density.
\begin{theorem}
\label{BD-Rdense}
Let $I^\ast=(I_k)_{k\in \N}$ be a sequence of intervals
$I_k=[a_k,a_k+b_k]\subseteq \R$ with $\limsup(b_k)_{k\in \N}=+\infty$
and $t_k:=\max(|a_k|,|a_k+b_k|)\geq k$ for every $k\in \N$. For each $k\in
\N$ let $\Lambda_k\subseteq I_k$ be a random subset with $|\Lambda_k|\neq
O(\log t_k)$. Then, almost surely, $\Lambda=\bigcup_{k\in \N} \Lambda_k$
is dense in $\R^\ap$.
\end{theorem}
The same argument   yields the following  general version of Theorem 3.1 in \cite{katz73}.
\begin{theorem}
\label{BD-Z}
Let $I^\ast=(I_k)_{k\in \N}$ be a sequence of intervals
$I_k=[n_k,n_k+m_k]\subseteq \Z$ with $\limsup (m_k)_{k\in \N}=+\infty$
and $t_k:=\max(|n_k|,|n_k+m_k|)\geq k$ for every $k\in \N$. For each $k\in
\N$ let $\Lambda_k\subseteq I_k$ be a random subset with $|\Lambda_k|\neq
O(\log t_k)$. Then, almost surely, $\Lambda=\bigcup_{k\in \N} \Lambda_k$
is dense in $\Z^{\ap}$.
\end{theorem}

\section{Bohr--dense sets with special properties}
\label{t-set}

The estimates of Section \ref{sec:appdense} can be easily used to find
Bohr--dense sets with special properties, as long as the conditions
imposed in Lemma \ref{firtsglobal} and Theorem \ref{BD-R} leave enough
room for a random subset to satisfy the required properties.

In this section  we focus on  interpolation properties of sets.
For a given algebra $\mathcal{A}\subseteq C(G,\C)$,  with $G$ a
topological group, a subset $X\subseteq G$ is said to be an \emph{
$\mathcal{A}$--interpolation set } if every bounded function $f\colon
X\to \C$ admits a continuous extension $\tilde{f}\colon G\to \C$ with
$\tilde{f}\in \mathcal{A}$.
We consider here a class of sets of interpolation for the algebra
$\wap(G) $ of weakly almost periodic functions on $G$. One of the
important features of this class is that the interpolation properties
of its members are not a part of its definition, which focuses on its
combinatorial side.
\begin{definition}
Let $G$ be a  topological group. A subset $E$ of $G$ is a t--set if for
every $g\in G$, $g\neq 0$, the intersection $E\cap (E+g)$ is relatively
compact.
\end{definition}
The class of $t$--sets was introduced by
Rudin in \cite{rudin59} where he  proved that  every function supported
on  a $t$--set of a  discrete group is automatically weakly almost
periodic, see \cite{filagali13} for further references on this topic
including the definition of weakly almost periodic function. In the
terminology of \cite{filagali13}, t--sets in LCA groups are (approximable)
$\wap(G)$--interpolation sets.

Sets of interpolation for the Fourier--Stieltjes algebra $B(G)$
(consisting of Fourier--Stieltjes transforms of measures on
$\widehat{G}$) are known as \emph{Sidon sets} and have been heavily
studied (see the monographs \cite{grahhare13} and \cite{lopeross}). Since
$B(G)\subseteq \wap(G)$,
we have that both  Sidon sets and  t--sets are sets  of interpolation
for the  algebra of weakly almost periodic functions. An interesting
observation is that every Sidon set can be decomposed as a finite
union of t--sets (see \cite[Corollary 6.4.7]{grahhare13}).

As already noted in \cite{katz73},  the random process of Theorems
\ref{BD-R} and \ref{BD-Z} cannot be adapted to yield
Sidon sets.  It is well known, see e.g.  \cite[Corollary
6.3.13]{grahhare13} that a length $N$ interval contains at most
$C_E\log N$ elements of a Sidon set $E$. This route has however
proved to be fruitful with other less demanding properties: Neuwirth
\cite{neuw99}, for instance, obtains dense subsets of $\Z^\ap$ that
are $\Lambda(p)$ for every $p$, and Li, Queff\'elec and Rodr\'\i guez-Piazza
\cite{liqueffrodr02} have obtained dense subsets of $\Z^\ap$ that are
$p$--Sidon for every $p>1$.
We do not need these concepts here and refer to
\cite{neuw99,liqueffrodr02,grahhare13} for their proper definitions. It
suffices to say that Sidon sets are $p$--Sidon  for every
$p>1$ and  $\Lambda(p)$--sets for every $p$.

While, as mentioned, our construction does not work with Sidon sets,
it does work with the important class of $t$--sets. We show in this
section that t--sets that are dense in $G^\ap$ do exist for $G=\Z$ and
$G=\R$ and that, indeed, random constructions in the spirit of Lemma
\ref{firtsglobal} lead almost surely to t--sets that are dense in $G^\ap$.

We first need a lemma that helps in recognizing t--sets.  For a subset
$\Lambda\subseteq \R$, we define its  \emph{step
length} as \[  \stl(\Lambda)=\inf\left\{|z-z^\prime|\colon z,z^\prime\in
\Lambda, z\neq z'\right\}.\]

\begin{lemma}\label{stl}
Let $([a_k,b_k])_{k\in \N}$ be a sequence of intervals in $\R$, and for
every $k\in \N$, let $\Lambda_k$ be a finite subset of $[a_k,b_k]$. If
\begin{enumerate}
\item the sequence of gaps $(a_{k+1}-b_k)_{k\in \N}$ is increasing and unbounded, and
\item \label{it:2} there exists $k_0$ such that
$\stl(\Lambda_k)>b_{k-1}-a_{k-1}$ for every $k\geq k_0$;
\end{enumerate}
then the set $\Lambda=\bigcup_{k\in \N} \Lambda_k$ is a  t--set.
\end{lemma}
\begin{proof}
Suppose there is  $0\neq t_0 \in  \R$ such that  $\Lambda \cap (\Lambda
+t_0)$ is unbounded. Since the sets $\Lambda_k$ are finite, there will
be  $k<k^\prime$, with $k>k_0$, $a_{k}-b_{k-1}>|t_0|$ and
$t_k\in \Lambda_k \cap (\Lambda +t_0)$,  $t_{k^\prime}\in  \Lambda_{k^
\prime} \cap (\Lambda +t_0)$.
Then
$t_0=t_k-t_1=t_{k^\prime}-t_2$ with  $t_{1}, t_{2}\in \Lambda$.
If $t_0>0$, using that $t_0<a_k-b_{k-1}<a_{k^\prime}-b_{k^\prime-1}$,
one has that $ b_{k-1}<t_1\leq t_k$ and that $b_{k^\prime-1}<t_2\leq
t_{k^\prime}$. A  symmetric  argument works when $t_0<0$. It follows
that   $t_1\in \Lambda_k$ and $t_2\in \Lambda_{k^\prime}$.
But then $\stl(\Lambda_{k^\prime})\leq
|t_{k^\prime}-t_2|=|t_0|<|a_k-b_{k-1}|$, a contradiction with  hypothesis
\eqref{it:2}.
\end{proof}

\begin{theorem}\label{tS-Z}
Let $I^\ast=(I_k)_{k\in \N}$ be a sequence of intervals
$I_k=[z_k,z_k+N_k]\subseteq \Z$ such that the sequence
of gaps $(z_{k+1}-z_k-N_k)_{k\in \N}$ is increasing and
unbounded. For each $k\in \N$ let $\Lambda_k\subseteq I_k$
be a random subset with $|\Lambda_k|=\ell_k$. If $\sum_{k\in
\N} \frac{\ell_k^2N_{k-1}}{N_k}<\infty$, then, almost surely,
$\Lambda=\bigcup_{k\in \N} \Lambda_k$ is a t--set.
\end{theorem}
\begin{proof} For $k\in \N$ we consider the event
$\mathcal{B}_k=\left\{\Lambda_k \colon \stl(\Lambda_k)\leq
N_{k-1}\right\}$. A rough estimate of  the probability of  this event
is obtained by observing that any choice of $\ell_k$ elements with
step length at most $N_{k-1}$ is \emph{witnessed} by elements in
$I_k$ at a distance of at most $N_{k-1}$.
If for each element $z$ in $I_k$ we find those elements
in $I_k$ larger than $z$ but within  a distance of at most $N_{k-1}$,
we see that there are at most $N_{k-1}$ witnesses containing
$z$. Since there are at most $\binom{N_k-1}{\ell_k-2}$
different subsets of $I_k$ of cardinality $\ell_k$ that contain a given witness,
we deduce altogether that
\[ \P\left(\mathcal{B}_k\right)\leq\frac{(N_k+1) N_{k-1}
\binom{N_k-1}{\ell_k-2}}{\binom{N_k+1}{\ell_k}}=\frac{(\ell_k-1)\ell_k
N_{k-1}}{N_k}\leq \frac{\ell_k^2N_{k-1}}{N_k}.\]
We conclude that $\sum_{k\in \N} \P\left(\mathcal{B}_k\right)<\infty$,
and the Borel--Cantelli Lemma
then shows that, almost surely, there exists $k_0 \in \N$ such that
$\Lambda_k \notin \mathcal{B}_k$
for every $k\geq k_0$. Lemma \ref{stl} then
proves that $\Lambda$ is a t--set.
\end{proof}
\begin{theorem}\label{tS-R}
Let $I^\ast=(I_k)_{k\in \N}$ be a sequence of intervals
$I_k=[z_k,z_k+N_k]\subseteq \R$ such that the sequence
of gaps $(z_{k+1}-z_k-N_k)_{k\in \N}$ is increasing and
unbounded. For each $k\in \N$ let $\Lambda_k\subseteq I_k$
be a random subset with $|\Lambda_k|=\ell_k$. If $\sum_{k\in
\N} \frac{\ell_k^2N_{k-1}}{N_k}<\infty$, then, almost surely,
$\Lambda=\bigcup_{k\in \N} \Lambda_k$ is a t--set.
\end{theorem}
\begin{proof}
Again, for $k\in \N$ we consider the event
$\mathcal{B}_k=\left\{\Lambda_k \colon \stl(\Lambda_k)\leq
N_{k-1}\right\}$. We first estimate the probability of $\mathcal{B}_k$
when $\ell_k=2$. For $\mathcal{B}_{0}:=\left\{ \{x,y\}\subseteq
I_k\colon  |x-y|\leq N_{k-1}\right\}$ and considering the pair $(x,y)$
to be uniformly distributed on the square $I_k\times I_k$, we obtain
\[
\P(\mathcal B_0)=1-2\int_{z_k+N_{k-1}}^{z_k+N_k}\int_{z_k}^{x-N_{k-1}}\frac{1}{N_k^2}\,dy\,dx=\frac{N_{k-1}(2N_k-N_{k-1})}{N_k^2}.
\]
Now, if a set $\Lambda_k$ consisting of $\ell_k$  points is  chosen
in $I_k$, for $\Lambda_k$ to be in $ \mathcal{B}_k$ it will be enough
that any pair $\{x,y\}$ of its elements satisfies $|x-y|\leq N_{k-1}$.
A rough estimate  is then
\begin{equation}
\P\left(\mathcal{B}_k\right)\leq \binom{\ell_k}{2}\cdot \P\left(\mathcal{B}_0\right)=
\frac{\ell_k(\ell_k-1)N_{k-1}}{2N_k^2}\left(2N_k-N_{k-1}\right)\leq
\frac{\ell_k^2N_{k-1}}{N_k}.
\end{equation}
We deduce that $\sum_{k\in \N} \P\left(\mathcal{B}_k\right)<\infty$, and
the Borel--Cantelli Lemma then shows that, almost surely, there exists
$k_0 \in \N$ such that $\Lambda_k \notin \mathcal{B}_k$ for every $k\geq
k_0$. Lemma \ref{stl} proves then that $\Lambda$ is a t--set.
\end{proof}

We now choose parameters in Theorem \ref{BD-R} so as to fit in Theorem \ref{tS-Z}.
\begin{theorem}\label{TSBD-Z}
Let $I^\ast=(I_k)_{k\in \N}$ be a sequence of intervals
$I_k=[L_k,2L_k]\subseteq \R$ with $L_k= (k!)^4$ for every $k\in \N$. For
each $k\in \N$ let $\Lambda_k\subseteq I_k$ be a random subset with
$|\Lambda_k|=k$, and let $\Lambda=\bigcup_{k\in \N} \Lambda_k$. Then,
almost surely, $\Lambda$ is a dense t--set in $\R^\ap$. If we choose
$\Lambda_k\subseteq \Z$, then, almost surely, $\Lambda$ is a dense t--set
in $\Z^\ap$.
\end{theorem}
\begin{proof}
This sequence of intervals satisfies the hypothesis of both
Theorems  \ref{BD-R} and \ref{tS-R}, or their analogs \ref{BD-Z} and
\ref{tS-Z}. Hence $\Lambda$ almost surely satisfies both conclusions.
\end{proof}\label{dens}
The sets of Theorems \ref{BD-Z} and \ref{TSBD-Z} have  asymptotic density zero.
Other examples of Bohr-dense subsets of $\Z$ are obtained in
\cite{blumeisehan73}. The sets in \cite{blumeisehan73}  satisfy the
condition
\[\lim_{N\to \infty}\frac{1}{N}\left|E_N\cap (E_N+k)\right|=1,\mbox{
for all  } k\in \Z,\]
where $E_N$ is the set consisting of the first  $N$ terms of $E$. They
are therefore  very far from being t--sets.

\end{document}